\documentclass{amsart}

\usepackage{amssymb}
\usepackage{hyperref}
\usepackage{color}

\newtheorem{theorem}{Theorem}
\newtheorem{lemma}[theorem]{Lemma}
\newtheorem{proposition}[theorem]{Proposition}

\newtheorem{conjecture}[theorem]{Conjecture}

\newtheorem*{Euler}{Euler's Formula}

\theoremstyle{definition}

\theoremstyle{remark}

\def\e{{\varepsilon}}

\begin{document}

\title{An improved upper bound for the bondage number of graphs 
on surfaces}
\author{Jia Huang}
\address{School of Mathematics, University of Minnesota, 
Minneapolis, MN 55455, USA}
\email{huang338@umn.edu}
\maketitle

\begin{abstract}
The bondage number $b(G)$ of a graph $G$ is the smallest number 
of edges whose removal from $G$ results in a graph with larger 
domination number. Recently Gagarin and Zverovich showed that, 
for a graph $G$ with maximum degree $\Delta(G)$ and embeddable 
on an orientable surface of genus $h$ and a non-orientable 
surface of genus $k$, $b(G)\leq\min\{\Delta(G)+h+2,\Delta+k+1\}$.
They also gave examples showing that adjustments of their proofs 
implicitly provide better results for larger values of $h$ and $k$.
In this paper we establish an improved explicit upper bound for $b(G)$, 
using the Euler characteristic $\chi$ instead of the genera $h$ and $k$, 
with the relations $\chi=2-2h$ and $\chi=2-k$. We show that 
$b(G)\leq\Delta(G)+\lfloor r\rfloor$ for the case $\chi\leq0$ 
(i.e. $h\geq1$ or $k\geq2$), where $r$ is the largest real root 
of the cubic equation $z^3+2z^2+(6\chi-7)z+18\chi-24=0$. Our proof 
is based on the technique developed by Carlson-Develin and 
Gagarin-Zverovich, and includes some elementary calculus as a new 
ingredient. We also find an asymptotically equivalent result 
$b(G)\leq\Delta(G)+\lceil\sqrt{12-6\chi\,}-1/2\rceil$ for 
$\chi\leq0$, and a further improvement for graphs with large girth.
\end{abstract}

\section{Introduction}

All graphs in this paper are finite, undirected, and without
loops or multiple edges. Let $G$ be a graph with vertex set 
$V(G)$ and edge set $E(G)$. Given a vertex $v$ in $G$, let 
$N(v)$ be the set of all neighbors of $v$ and let $d(v)=|N(v)|$ 
be the degree of $v$. The maximum and minimum vertex degrees 
of $G$ are denoted by $\Delta(G)$ and $\delta(G)$.

A \emph{dominating set} for a graph $G$ is a subset $D\subseteq V$ of 
vertices such that every vertex not in $D$ is adjacent to
at least one vertex in $D$. The minimum cardinality of a dominating 
set is called the \emph{domination number} of $G$.
The concept of domination in graphs has many applications 
in a wide range of areas within the natural and social sciences.

The \emph{bondage number} $b(G)$ of a graph $G$ is defined as the 
smallest number of edges whose removal from $G$ results in a graph 
with larger domination number. The bondage number of $G$ was introduced 
in \cite{BHNS,FJKR}, measuring to some extent the reliability of 
the domination number of $G$ with respect to edge removal from $G$
(which may corresponds to link failure in communication networks).

In general it is NP-hard to determine the bondage number $b(G)$ 
(see Hu and Xu~\cite{HuXu}), and thus useful to find bounds for it. 

\begin{lemma}[Hartnell and Rall~\cite{HR}]\label{lem:HR}
For any edge $uv$ in a graph $G$, we have
\[
b(G) \leq d(u) + d(v)-1-|N(u) \cap N(v)|.
\] 
In particular, $b(G)\leq \Delta(G) + \delta(G)-1$.
\end{lemma}

The following two conjectures are still open. 

\begin{conjecture}[Teschner~\cite{Teschner}]\label{conj1}
For any graph $G$, $b(G)\leq \frac32 \Delta(G)$.
\end{conjecture}

\begin{conjecture}[Dunbar-Haynes-Teschner-Volkmann~\cite{DHTV}]
\label{conj2}
For any planar graph $G$, $b(G)\leq \Delta(G)+1$.
\end{conjecture}

On the way of attacking Conjecture~\ref{conj2},
Kang and Yuan~\cite{KangYuan} had the following result.

\begin{theorem}[Kang and Yuan~\cite{KangYuan}]
For any planar graph $G$, $b(G)\leq\min\{\Delta(G)+2,8\}$.
\end{theorem}

A simpler proof for the above theorem was later given by 
Carlson and Develin~\cite{CarlsonDevelin}, whose ideas
were further extended by Gagarin and Zverovich~\cite{GagarinZverovich} 
to establish a nice upper bound for arbitrary graphs, a step 
forward towards Conjecture~\ref{conj1}. To state this result
we first recall some basic facts about graphs on surfaces below;
the readers are referred to Mohar and Thomassen~\cite{MoharThomassen}
for more details.

Throughout this paper a \emph{surface} means a connected compact 
Hausdorff topological space which is locally homeomorphic to 
an open disc in $\mathbb R^2$. The classification theorem 
for surfaces~\cite[Theorem 3.1.3]{MoharThomassen} states that, 
any surface $S$ is homeomorphic to either $S_h$ ($h\geq0$) which 
is obtained from a sphere by adding $h$ handles, or $N_k$ 
($k\geq1$) which is obtained from a sphere by adding $k$ crosscaps. 
In the former case $S$ is an \emph{orientable surface of genus $h$}, 
and in the latter case $S$ is a \emph{non-orientable surface of genus $k$}.
For example, the torus, the projective plane, and the Klein 
bottle are homeomorphic to $S_1$, $N_1$, and $N_2$, respectively.
The \emph{Euler characteristic} of $S$ is defined as
\[
\chi(S)=\begin{cases}
         2-2h, & S\cong S_h,\\
         2-k, & S\cong N_k.
        \end{cases}
\]

Any graph $G$ can be embedded on some surface $S$, i.e. it can be 
drawn on $S$ with no crossing edges; in addition, the surface $S$ 
can be taken to be either orientable or non-orientable. Denote by 
$\chi(G)$ the largest integer $\chi$ for which $G$ admits an 
embedding on a surface $S$ with $\chi(S)=\chi$. For example, $G$ 
is planar if and only if $\chi(G)=2$.

\begin{theorem}[Gagarin and Zverovich~\cite{GagarinZverovich}]
\label{thm:GZ}
Let $G$ be a graph embeddable on an orientable surface of 
genus $h$ and a non-orientable surface of genus $k$. Then
$b(G)\leq\min\{\Delta(G)+h+2,\Delta(G)+k+1\}$.
\end{theorem}

According to Theorem~\ref{thm:GZ}, if $G$ is planar ($h=0$, $\chi=2$) or 
can be embedded on the real projective plane ($k=1$, $\chi=1$), then $b(G)\leq\Delta(G)+2$. For larger values of $h$ and $k$, it was mentioned 
in \cite{GagarinZverovich} that improvements of Theorem~\ref{thm:GZ} 
can be achieved by adjusting its proof - for example, with the same 
assumptions as above, 
\begin{equation}\label{eq:improv}
b(G)\leq\Delta(G)+
\begin{cases}
h+1, & {\rm if}\ h\geq8,\\
h, & {\rm if}\ h\geq11,\\
k, & {\rm if}\ k\geq3,\\
k-1, & {\rm if}\ k\geq6.
\end{cases}
\end{equation}
The goal of this paper is to establish the following explicit improvement of Theorem~\ref{thm:GZ}.

\begin{theorem}\label{thm:H1}
Let $G$ be a graph embedded on a surface whose Euler
characteristic $\chi$ is as large as possible.
If $\chi\leq 0$ then $b(G)\leq\Delta(G)+\lfloor r\rfloor$,
where $r$ is the largest real root of the following 
cubic equation in $z$:
\[
z^3+2z^2+(6\chi-7)z+18\chi-24 = 0.
\]
In addition, if $\chi$ decreases then $r$ increases.
\end{theorem}

Our proof for Theorem~\ref{thm:H1} is based on the technique developed by 
Carlson-Develin and Gagarin-Zverovich, and includes some elementary 
calculus (mainly the intermediate value theorem and the mean value 
theorem) as a new ingredient.

We will show that $r$ is the unique positive root of the above cubic
equation when $\chi\leq0$. The explicit formula for $r$ is complicated 
and will be given in Section 3. However, we have a simpler result which 
turns out to be asymptotically equivalent to Theorem~\ref{thm:H1}.

\begin{theorem}\label{thm:H2}
Let $G$ be a graph embedded on a surface whose Euler
characteristic $\chi$ is as large as possible. If $\chi\leq0$
then $b(G)<\Delta(G)+\sqrt{12-6\chi}+1/2$, or equivalently,
$b(G)\leq\Delta(G)+\lceil\sqrt{12-6\chi}-1/2\rceil$.
\end{theorem}

We will prove Theorem~\ref{thm:H1} and Theorem~\ref{thm:H2}
in Section 2. Then some remarks will be given in Section 3,
including the explicit formula for $r$, a comparison 
of Theorem~\ref{thm:GZ} (for $\chi\leq0$), Theorem~\ref{thm:H1},
and Theorem~\ref{thm:H2}, and a further improvement of
Theorem~\ref{thm:H1} for graphs with large girth.

\section{Proofs for the main results}

Let $G$ be a connected graph which admits an embedding on 
a surface $S$ whose Euler characteristic $\chi$ is as large 
as possible. By Mohar and Thomassen~\cite[\S3.4]{MoharThomassen},
this embedding of $G$ on $S$ can be taken to be a 
\emph{2-cell embedding}, namely an embedding with all faces 
homeomorphic to an open disk. 

\begin{Euler}(c.f. \cite{MoharThomassen})
Suppose that a graph $G$ with vertex set $V(G)$ and edge set $E(G)$
admits a $2$-cell embedding on a surface $S$, and let $F(G)$ be the set
of faces in this embedding. Then
\[
|V(G)|-|E(G)|+|F(G)|=\chi(S).
\]
\end{Euler}

Every edge $uv$ in the 2-cell embedding of $G$ on $S$ appears on 
the boundary of either two distinct faces $F\ne F'$ or a unique 
face $F=F'$; in the former case $uv$ occurs exactly once on the 
boundary of each of the two faces $F$ and $F'$, while in the latter 
case $uv$ occurs exactly twice on the boundary of the face $F=F'$. 
Let $m$ and $m'$ be the number of edges on the boundary of 
$F$ and $F'$, whether or not $F$ and $F'$ are distinct. 
For instance, a path $P_n$ with $n$ vertices is embedded on a 
sphere with only one face, and for any edge in $P_n$ we have 
$m=m'=2(n-1)$. We may assume that $m$ and $m'$ are at least $3$, 
since $m\leq 2$ or $m'\leq 2$ implies $G=P_2$ which is trivial.
Following Gagarin and Zverovich~\cite{GagarinZverovich},
we define the \emph{curvature} of $uv$ to be
\[
w(uv)=\frac{1}{d(u)}+\frac{1}{d(v)}-1+\frac{1}{m}
+\frac{1}{m'}-\frac{\chi}{|E(G)|}.
\]
It follows from Euler's formula that
\begin{equation}\label{eq:weight}
\sum_{uv\in E(G)} w(uv)=|V(G)|-|E(G)|+|F(G)|-\chi=0.
\end{equation}

\begin{lemma}\label{lem:main}
Let $G$ be a connected graph embedded on a surface whose 
Euler characteristic $\chi$ is as large as possible. 
Then $b(G)<\Delta(G)+z$, if $\chi\leq 0$ and 
$z\geq0$ satisfy
\begin{equation}\label{ineq1}
z^2-z+4\chi-6 > 0,
\end{equation}
\begin{equation}\label{ineq2}
5z^3+6z^2+(24\chi-31)z+48\chi-70 > 0,
\end{equation}
\begin{equation}\label{ineq3}
z^3+2z^2+(6\chi-7)z+18\chi-24 > 0.
\end{equation}
\end{lemma}

\begin{proof}
Suppose to the contrary that $b(G)\geq \Delta(G)+z$. 
Let $uv$ be an arbitrary edge in $G$. Assume $d(u)\leq d(v)$
and $m\leq m'$, without loss of generality. By Lemma~\ref{lem:HR},
\[
\Delta(G)+z\leq b(G)\leq d(u)+d(v)-1-|N(u)\cap N(v)|
\]
and thus
\[
0\leq|N(u)\cap N(v)|\leq d(u)+d(v)-1-\Delta(G)-z
\leq d(u)-1-z. 
\]
It follows that $d(u)\geq z+1$, and so $\delta(G)\geq z+1$ since 
the edge $uv$ is arbitrary.
We distinguish three cases below for the value of $d(u)$.

If $d(u)=\lceil z\rceil+1$ then $|E(G)|\geq (z+1)(z+2)/2$,
and $|N(u)\cap N(v)|=0$, which implies $m'\geq m\geq4$.
Thus 
\begin{eqnarray*}
w(uv) &\leq & \frac 2{z+1}+\frac14+\frac14 -1
-\frac{2\chi}{(z+1)(z+2)} \\ 
&=& -\frac{z^2-z+4\chi-6}{2(z+1)(z+2)} <0
\end{eqnarray*}
where the last inequality follows from (\ref{ineq1}).

If $d(u)=\lceil z\rceil+2$ then $|E(G)|\geq (2(z+2)+(z+1)^2)/2$,
and $|N(u)\cap N(v)|\leq 1$, which implies $m'\geq 4$, $m\geq 3$. Thus
\begin{eqnarray*}
w(uv) & \leq & \frac2{z+2}+\frac14+\frac13-1-
\frac{2\chi}{2(z+2)+(z+1)^2} \\
&=& -\frac{5z^3+6z^2+(24\chi-31)z+48\chi-70}
{12(z+2)(z^2+4z+5)}<0
\end{eqnarray*}
where the last inequality follows from (\ref{ineq2}).

If $d(u)\geq\lceil z\rceil+3$ then 
$|E(G)|\geq (2(z+3)+(z+1)(z+2))/2$, and $m'\geq m\geq 3$. Thus
\begin{eqnarray*}
w(uv)&\leq& \frac2{z+3}+\frac13+\frac13-1
-\frac{2\chi}{2(z+3)+(z+1)(z+2)}\\
&=&-\frac{z^3+2z^2+(6\chi-7)z+18\chi-24}
{3(z+3)(z^2+5z+8)}<0
\end{eqnarray*}
where the last inequality follows from (\ref{ineq3}).

Therefore $w(uv)<0$ for all edges $uv$ in $G$.
This contradicts Equation (\ref{eq:weight}).
\end{proof}

\begin{lemma}\label{lem:equiv}
Let $z\geq0$ and $\chi\leq0$. Then the inequalities
(\ref{ineq1},\ref{ineq2},\ref{ineq3}) hold if and only if
$z>r$, where $r$ is the largest real root of the
following cubic equation in $z$:
\[
z^3+2z^2+(6\chi-7)z+18\chi-24 = 0.
\]
\end{lemma}

\begin{proof}
Fix a $\chi\leq0$ and consider the left hand side of
(\ref{ineq1},\ref{ineq2},\ref{ineq3}) as polynomials in $z$:
\begin{eqnarray*}
A(z) &=& z^2-z+4\chi-6, \\
B(z) &=& 5z^3+6z^2+(24\chi-31)z+48\chi-70, \\
C(z) &=& z^3+2z^2+(6\chi-7)z+18\chi-24.
\end{eqnarray*}
We first show that the largest real root $r$ of $C(z)$ is larger 
than or equal to the real roots of $A(z)$ and $B(z)$,
by using the intermediate value theorem and the limits
\[
\lim_{z\to\infty}A(z)=
\lim_{z\to\infty}B(z)=
\lim_{z\to\infty}C(z)=\infty.
\]

The polynomial $A(z)$ has two roots
\[
z_1=\frac12+\frac12\sqrt{25-16\chi}>0,\quad
z_2=\frac12-\frac12\sqrt{25-16\chi}<0.
\]
Substituting $z_1$ in $C(z)$ gives
\[
C(z_1)= (\chi+1)\sqrt{25-16\chi}+7\chi-5
\]
which is negative if $\chi\leq -1$ and is $0$ if $\chi=0$.
By the intermediate value theorem,
$C(z)$ has a real root larger than or equal to
$z_1$, and so $r\geq z_1>z_2$.

Next consider $B(z)$. If $\chi=0$ then $B(z)$ has a 
unique real root $14/5$ and $C(z)$ has a unique real 
root $3>14/5$. Assume $\chi\leq -1$ below. 
Then $B(3)=120\chi+26<0$. Applying the intermediate 
value theorem to $B(z)$ gives the existence of real root(s) 
of $B(z)$ in $(3,\infty)$; let $z_3$ be the largest one. Then
\begin{eqnarray*}
B(z_3)-4C(z_3) &=& z_3^3-2z_3^2-3z_3+26-24\chi \\
&=& z_3(z_3+1)(z_3-3) + 26-24\chi >0
\end{eqnarray*}
which implies $C(z_3)<0$. Again by the intermediate value
theorem, $C(z)$ has a root larger than $z_3$, and thus
$r>z_3$.

Therefore $r$ is larger than or equal to any real root of
$A(z)$ and $B(z)$. It follows that $A(z)$, $B(z)$, 
and $C(z)$ are all positive for all $z>r$; otherwise the 
intermediate value theorem would imply that $A(z)$, $B(z)$, 
or $C(z)$ has a root larger than $r$, a contradiction.

Conversely, suppose that $A(z)$, $B(z)$, and $C(z)$ are all 
positive at some point $z=s\geq0$. Then $s\ne r$ since $C(r)=0$.
If $s<r$, then there exists a point $t$ in $(s,r)$ such that 
\[
C'(t)=3t^2+4t+6\chi-7<0
\] 
by the mean value theorem. It follows that
\[
A(t)=\frac23 C'(t) - t^2-\frac{11}3 t-\frac43<0.
\]
We have seen that the upward parabola $A(z)$ has two roots 
$z_1>0$ and $z_2<0$. Then $A(s)>0$ and $s\geq0$ imply $s>z_1$,
and $t>s$ implies $A(t)>0$, which contradicts what we found
above. Hence $s>r$.
\end{proof}

\begin{proof}[Proof of Theorem~\ref{thm:H1}]
Let $r(\chi)$ be the largest root of 
$C(z;\chi)=z^3+2z^2+(6\chi-7)z+18\chi-24$ for $\chi\leq0$.
We first show that $r(\chi)$ increases as $\chi$ decreases.
We have seen in the proof of Lemma~\ref{lem:equiv} that
$r(\chi)\geq3$. It follows from
\[
C(z;\chi)-C(z;\chi-1)=6z+18 
\]
that $C(r(\chi);\chi-1)=-6r(\chi)-18<0$. By the 
intermediate value theorem, $C(z,\chi-1)$ has a root 
larger than $r(\chi)$, and thus its largest root 
$r(\chi-1)$ is also larger than $r(\chi)$.

Now we prove the upper bound for $b(G)$. If $G$ has multiple 
components $G_1,\ldots,G_\ell$, then $\chi\leq\chi_i=\chi(G_i)$ for 
all $i$, since an embedding of $G$ on a surface $S$ automatically
includes an embedding of $G_i$ on $S$. It follows from the definition 
that $b(G)=\min\{b(G_1),\ldots,b(G_\ell)\}$. By Theorem~\ref{thm:GZ}, 
we define $r(1)=r(2)=2$ which is always less than $r(\chi)$ for
$\chi\leq0$. If we could establish our upper bound 
for connected graphs, then 
\[
b(G)\leq b(G_i)\leq\Delta(G_i)+\lfloor r(\chi_i)\rfloor
\leq \Delta(G)+\lfloor r(\chi)\rfloor
\]
and we are done. Therefore we assume $G$ is connected below. 

It follows from Lemma~\ref{lem:main} and Lemma~\ref{lem:equiv} that
$b(G)<\Delta(G)+z$ for all $z>r(\chi)$. Writing $z=r(\chi)+\e$ and 
taking the one-sided limit as $\e\to0^+$ gives $b(G)\leq\Delta(G)+r(\chi)$.
The result then follows immediately from $b(G)$ being an integer.
\end{proof}

\begin{proof}[Proof of Theorem~\ref{thm:H2}]
We can assume $G$ is connected for the same reason as 
discussed in the proof of Theorem~\ref{thm:H1}.
Let $z=\sqrt{12-6\chi}+1/2$. Then for $\chi\leq0$ we have
\[
A(z) = \frac{23}4-2\chi >0,
\]
\[
B(z) = (\sqrt{12-6\chi\,})^3+\frac{107}4 \sqrt{12-6\chi\,}+\frac{629}8-21\chi >0,
\]
\[
C(z) = \frac{31}4\sqrt{12-6\chi\,}+\frac{121}8 > 0.
\]
The result follows immediately from Lemma \ref{lem:main}
and $b(G)$ being an integer.
\end{proof}

\section{Remarks}

Using the cubic formula (c.f. M. Artin~\cite{Artin}) one can show
that the largest real root of $C(z)$ is
\begin{multline*}
r=\frac{25-18\chi}{3\left(
253-189\chi+3\sqrt{5376-6876\chi+1269\chi^2+648\chi^3}
\right)^{1/3}}\\
+\frac13\left(253-189\chi+3\sqrt{5376-6876\chi+1269\chi^2+648\chi^3}
\right)^{1/3}-\frac23.
\end{multline*}
Some explanations are needed to make this formula work. Let 
$f=5376-6876\chi+1269\chi^2+648\chi^3$. If $-4\leq\chi\leq0$ then
$f\geq0$ and the formula works within $\mathbb R$, giving the 
unique real root of $C(z)$. If $\chi\leq -5$ then $f<0$ and we 
need to allow complex numbers when applying the formula. We may 
take $\sqrt f$ to be either of the two square roots of $f$.
Then there are three choices for the cubic roots of 
$253-189\chi+3\sqrt f$, giving three distinct 
real roots of $C(z)$, and we take $r$ to be the largest one. 

One can also see that $r$ is the unique positive root of $C(z)$ when
$\chi\leq0$, since $C(0)=18\chi-24<0$ for $\chi\leq0$ and 
$C''(z)=6z+4>0$ for $z>0$.

Next we consider Theorem~\ref{thm:H2}. By Lemma~\ref{lem:equiv},
$r<\sqrt{12-6\chi}+1/2$. Hence Theorem~\ref{thm:H2} is implied by
Theorem~\ref{thm:H1}. We show that these two results are asymptotically 
equivalent, i.e. 
\begin{equation}\label{eq:lim}
\lim_{\chi\to-\infty} \frac{\sqrt{12-6\chi}+1/2}{r}=1.
\end{equation}
In fact, for any $\e\in(0,1)$, substituting $z=(1-\e)\left(\sqrt{12-6\chi}+1/2\right)$ in $C(z)$ gives
\begin{multline*}
\frac{121}{8}-\frac{799\e}{8}+\frac{631\e^2}{8}-\frac{145\e^3}{8} 
+3\e(3\e^2-13\e+16)\chi \\
+\left(\frac{31}{4}-\frac{141\e}{4}+\frac{161\e^2}{4}
-\frac{51\e^3}{4} + 6\e(\e^2-3\e+2)\chi\right)\sqrt{12-6\chi}.
\end{multline*}
Since $3\e^2-13\e+16>0$ and $\e^2-3\e+2>0$, the above expression 
is negative when $\chi$ is small enough. It follows from the 
intermediate value theorem that $r>(1-\e)\left(\sqrt{12-6\chi}+1/2\right)$.
Therefore (\ref{eq:lim}) holds.

As pointed out by Gagarin and Zverovich~\cite{GagarinZverovich}, 
if $\chi=\chi(G)\leq0$ then $\delta(G)\leq \left\lfloor\frac{5+\sqrt{49-24\chi}}2\right\rfloor$ 
(see Sachs~\cite{Sachs}, for example). 
It follows immediately from Lemma~\ref{lem:HR} that
\[
b(G)\leq\Delta(G)+\left\lfloor\frac{3+\sqrt{49-24\chi}}2\right\rfloor.
\]
Our Theorem~\ref{thm:H2} improves this by $1$ or $2$, since 
\[
\sqrt{12-6\chi}+1/2=\frac{1+\sqrt{48-24\chi}}{2}.
\]

Now consider the results given in \cite{GagarinZverovich}. One can 
prove Theorem~\ref{thm:GZ} for $\chi\leq0$ by showing that
$z=h+3=4-\chi/2$ (for even $\chi\leq0$ achieved by embeddings on 
orientable surfaces) and $z=k+2=4-\chi$ (for all $\chi\leq0$ achieved 
by embeddings on non-orientable surfaces) satisfy the inequalities 
(\ref{ineq1},\ref{ineq2},\ref{ineq3}). By Lemma~\ref{lem:equiv},
Theorem~\ref{thm:H1} implies Theorem~\ref{thm:GZ} for $\chi\leq0$. 
Similarly Theorem~\ref{thm:H1} implies (\ref{eq:improv}).

We give a table below to show our upper bound for $\chi=0,-1,\ldots,-21$.
\[
\begin{tabular}{|c|ccccccccccc|}
\hline
$\chi$ & 0 & -1 & -2 & -3 & -4 & -5 & -6 & -7 & -8 & -9 & -10  \\
\hline 
$\lfloor r \rfloor$ & 3 & 3 & 4 & 5 & 5 & 6 & 6 & 7 & 7 & 8 & 8\\
\hline\hline
$\chi$ &-11&-12&-13& -14 & -15 & -16 & -17 & -18 & -19 & -20 & -21\\
\hline 
$\lfloor r\rfloor$ &8 & 9 & 9 & 9 & 10 & 10 & 10 & 11 & 11 & 11 &11\\
\hline
\end{tabular}
\] 


Our result can be further improved when the graph $G$ has
large \emph{girth} $g(G)$, defined as the length of the 
shortest cycle in $G$. If $G$ has no cycle then $g(G)=\infty$,
and $b(G)\leq 2$ by \cite{BHNS}.

\begin{proposition}
Let $G$ be a graph embedded on a surface whose Euler characteristic 
$\chi$ is as large as possible. If $\chi\leq0$ and $g=g(G)<\infty$, 
then $b(G)\leq\Delta(G)+\lfloor s \rfloor$
where $s$ is the larger root of the quadratic polynomial 
$A(z)=(g-2)z^2+(g-6)z+2\chi g-2g-4$, i.e.
\[
s=\frac{\sqrt{8g(2-g)\chi+(3g-2)^2}-(g-6)}{2(g-2)}.
\]
\end{proposition}

\begin{proof}
Assume $G$ is connected for the same reason as in the proof of
Theorem~\ref{thm:H1}. It suffices to show that $b(G)<\Delta(G)+z$ for
all $z\geq 0$ with $A(z)>0$; the result follows from writing $z=s+\e$ and taking the one-sided limit as $\e\to0^+$.

Suppose to the contrary that $b(G)\geq \Delta(G)+z$. Then
$\delta(G)\geq z+1$ by Lemma~\ref{lem:HR}. Let $uv$ be an 
arbitrary edge in $G$. It is clear that $m,m'\geq g$, and thus
\begin{eqnarray*}
w(uv) &\leq & \frac 2{z+1}+\frac2g-1-\frac{2\chi}{(z+1)(z+2)} \\ 
&=& -\frac{(g-2)z^2+(g-6)z+2\chi g-2g-4}{g(z+1)(z+2)}<0
\end{eqnarray*}
where the last inequality follows from $A(z)>0$.
This contradicts Equation (\ref{eq:weight}).
\end{proof}

For example, we have 
\[
b(G)\leq\Delta(G)+
\begin{cases}
2, & {\rm if}\ \chi=0,\ g\geq 5,\\
1, & {\rm if}\ \chi=0,\ g\geq 7,\\
2, & {\rm if}\ \chi=-1,\ g\geq 6,\\
3, & {\rm if}\ \chi=-2,\ g\geq 5, \\
2, & {\rm if}\ \chi=-2,\ g\geq 7.
\end{cases}
\]

\section*{Acknowledgement}

The author thanks the anonymous referees for their valuable 
suggestions and comments. He also thanks Rong Luo, Victor Reiner 
and Arthur White for helpful conversations and email correspondence.


\begin{thebibliography}{30}

\bibitem{Artin}
M. Artin, Algebra, Prentice Hall, 1991.

\bibitem{BHNS}
D. Bauer, F. Harary, J. Nieminen, C.L. Suffel, Domination alteration
sets in graphs, Discrete Math. 47 (1983) 153--161.

\bibitem{CarlsonDevelin}
K. Carlson and M. Develin, On the bondage number of planar and directed
graphs, Discrete Math. 306 (2006) 820--826.

\bibitem{DHTV}
J.E. Dunbar, T.W. Haynes, U. Teschner and L. Volkmann, Bondage, insensitivity
and reinforcement, in: T.W. Haynes, S.T. Hedetniemi, P.J. Slater
(Eds.), Domination in Graphs: Advanced Topics, Marcel Dekker, New
York, 1998, pp. 471--489.

\bibitem{FJKR}
J.F. Fink, M.J. Jacobson, L.F. Kinch and J. Roberts, The bondage number
of a graph, Discrete Math. 86 (1990) 47--57.

\bibitem{GagarinZverovich}
A. Gagarin and V. Zverovich, Upper bounds for the bondage number 
of graphs on topological surfaces, to appear in Discrete Math..

\bibitem{HR}
B.L. Hartnell and D.F. Rall, Bounds on the bondage number of a graph,
Discrete Math. 128 (1994) 173--177.

\bibitem{HuXu}
F.-T. Hu and J.-M. Xu, On the complexity of the bondage and reinforcement 
problems, to appear in Journal of Complexity.

\bibitem{KangYuan}
L. Kang and J. Yuan, Bondage number of planar graphs, Discrete Math.
222 (2000) 191--198.

\bibitem{MoharThomassen}
B. Mohar and C. Thomassen, Graphs on surfaces, John Hopkins University Press, 
Baltimore, 2001.

\bibitem{Sachs}
H. Sachs, Einf\"uhrung in die Theorie der endlichen Graphen, Teil II,
Teubner, Leipzig, 1972 (in German).

\bibitem{Teschner}
U. Teschner, A new upper bound for the bondage number of graphs with
small domination number, Australas. J. Combin. 12 (1995) 27--35.

\end{thebibliography}
\end{document}